\documentclass[3p, a4paper, 11pt]{elsarticle}

\usepackage{amssymb}
\usepackage{amsmath}
\usepackage{amsthm}
\usepackage{algorithm}
\usepackage{algorithmic}
\usepackage{float}
\usepackage[capitalize]{cleveref}
\usepackage{graphicx}
\usepackage{subcaption}
\usepackage{booktabs}
\usepackage[most]{tcolorbox} 
\journal{XXX}

\begin{document}
\begin{frontmatter}

\title{Riemannian optimal reduction for linear systems with quadratic outputs} 

\author[author1]{Xiaolong Wang}
\author[author1]{Chenglong Liu}
\affiliation[author1]{
organization={School of Mathematics and Statistics},
addressline={Northwestern Polytechnical University}, 
city={Xi'an},
postcode={710129}, 
state={Shaanxi},
country={China}}

\begin{abstract}
This paper presents an $H_2$-optimal model order reduction (MOR) method for linear 
systems with quadratic outputs based on Riemannian optimization. The $H_2$-optimal 
MOR is 
formulated as an optimization problem in which the optimization 
variables are selected directly as the coefficient matrices of reduced models. 
The product manifold is defined properly to impose the stability condition for 
reduced 
models. By exploiting the geometric properties of the product manifold, we derive an 
explicit formula for Riemannian gradient of the objective function, and then a 
limited-memory Riemannian BFGS method is adopted to solve the resulting optimization 
problem 
iteratively. In contrast to selecting projection matrices, optimizing coefficient 
matrices of reduced models reduces the amount of variables dramatically. 
Numerical simulation results demonstrate that reduced models  
accurately approximate the original system and exhibit superior performance in terms 
of $H_2$ error, which confirms the effectiveness of the proposed algorithm.

\end{abstract}
\end{frontmatter}




\section{Introduction}
\label{section1}

Accurate modeling of large-scale complex systems is crucial in fields such as control engineering and signal processing.
However, these high-fidelity, high-dimensional models often pose significant 
challenges in the simulation, including immense computational complexity and 
substantial storage requirements, hindering their practical application.
Model order reduction (MOR) addresses these issues by constructing low-dimensional 
approximate models.
The primary goal of MOR is to preserve the essential input-output dynamics and key 
characteristics of original systems, thereby significantly reducing computational 
and storage demands and enhancing tractability and efficiency.
A variety of MOR techniques have been established.
Balanced truncation (BT) method is a standard benchmark for evaluating 
the accuracy and speed of new MOR algorithms \cite{Moore1981}.
Moment matching based on Krylov subspaces has been also well studied in the past 
decades \cite{Astolfi2010}.
For further details on these methods, see \cite{Benner2017,Antoulas2005}.
\par

In some engineering problems, the dynamical systems yield observables that are 
formulated 
as the combinations of the states and the sample variances or deviations, resulting 
in models with quadratic outputs.
Researchers have proposed numerous methods for MOR of linear systems with quadratic 
outputs (LQO).
Early, the multi-input single-output (MISO) LQO system was reformulated as a linear 
multi-input multi-output (MIMO) system. 
This reformulation enables the application of the standard MOR techniques, such as 
BT \cite{VanBeeumen2010} and moment-matching methods \cite{Benner2015}, to LQO 
systems.
More recently, some researchers focus on exploiting the direct MOR methods for LQO 
systems, avoiding any lifting or linearization procedures. Peter Benner et al. 
proposed a BT algorithm for LQO systems based on a specially defined quadratic 
output 
observability Gramian \cite{Benner2021}.
This kind of methods guarantees the asymptotic stability of reduced models and 
provides a 
posteriori $H_2$ error bound, establishing a foundation for subsequent research in 
this area \cite{Gosea2019,Gosea2022}.
\par

The $H_2$-norm is a standard metric in the analysis and synthesis of dynamical 
systems.
The $H_2$ optimal MOR aims to minimize the $H_2$-norm error between the 
original systems and reduced models \cite{Xu2013,Yan1999}.
However, the primitive minimization for the $H_2$-norm may result in unstable 
reduced models, even if the original system is stable.
A prominent approach is to reformulate the optimization problem with stability 
conditions as an unconstrained optimization on Riemannian manifolds, where the 
manifold structure inherently ensures the required constraints \cite{Absil2008}.
A convergent $H_2$-optimal MOR algorithm based on product manifold geometry has been 
provided for linear systems \cite{Sato2016}, while a 
bilateral iterative algorithm leveraging Grassmann manifolds has been given in 
\cite{Zeng2015}. The $H_2$-optimal MOR for bilinear systems also is discussed in 
\cite{Xu2019} and \cite{Xu2015} based on Stiefel and Grassmann manifolds. More 
extensions based on Riemannian optimization to other systems with specific 
structures can be 
found in \cite{Yang2017,Jiang2019}.
\par

In this paper we consider the $H_2$-optimal MOR of LQO systems.
We first formulate the MOR procedure as an optimization problem to minimize the 
$H_2$-norm error between the original system and reduced models.
However, the explicit constraint to preserve the stability of original systems makes 
the optimization problem intractable.
Consequently, we equivalently reformulate the optimization problem as an 
unconstrained optimization on a specific product manifold, which is composed of the 
vector space of skew-symmetric matrices, the manifold of symmetric positive definite 
(SPD) matrices, two Euclidean spaces, and the vector space of symmetric matrices.
The Riemannian gradient of the objective function on product manifold is obtained in 
theory with the aid of the geometric properties of Riemannian manifold.
Then a Riemannian BFGS method is adopted to solve the optimization problem 
iteratively, leading to stable reduced models that inherent the quadratic output 
structure of the original system.
\par

This paper is organized as follows.
Section 2 introduces LQO systems and the $H_2$ norm.  
In Section 3, the $H_2$-optimal 
MOR is reformulated as an $H_2$-optimization problem on the product manifold. The 
essential Riemannian geometric properties of the product 
manifold are also presented, which are crucial for our framework.
In Section 4, the explicit expression for the Riemannian gradient of the objective 
function is derive in theory, and the Riemannian BFGS is employed to solve the 
optimization problem on product manifold, leading to an efficient MOR algorithm.
Section 5 presents the numerical results to demonstrate the feasibility and 
effectiveness of the proposed algorithm. Finally, some conclusions are given in 
Section 6. 

\subsection*{Notations}
We denote the sets of real numbers by $\mathbb{R}$, the $n$ dimensional real vector space by $\mathbb{R}^{n}$, and the space of $n \times n$ real matrices by $\mathbb{R}^{n \times n}$.
The vector space of skew-symmetric matrices is referred as $\text{Skew}(n) \subset 
\mathbb{R}^{n \times n}$, the vector space of symmetric matrices as $\text{Sym}(n) 
\subset \mathbb{R}^{n \times n}$, and the manifold of SPD matrices as 
$\text{Sym}_+(n) \subset \mathbb{R}^{n \times n}$.
We denote the set of stable matrices by $\mathbb{S}^{n\times 
n}:=\{A\in\mathbb{R}^{n\times n} \mid \text{all eigenvalues of } A \text{ have the 
negative real parts}\}$.
The trace of $A$, $\text{tr}(A)$, is the sum of its diagonal elements.
The symmetric and skew-symmetric parts of $A$ are defined as $\text{sym}(A) := 
\frac{A+A^{\top}}{2}$ and $\text{skew}(A) := \frac{A-A^{\top}}{2}$, respectively.
$\exp(A)$ denotes the standard matrix exponential of $A$.
For a vector $v \in \mathbb{R}^n$, $||v||$ denotes the Euclidean norm.
For a matrix $A \in \mathbb{R}^{n \times n}$, $||A||$ and $||A||_F$ denote the 
induced $2$-norm (spectral norm) and the Frobenius norm, respectively.
Given a measurable function $f: [0, \infty) \to \mathbb{R}^n$, its $L_2$ and 
$L_\infty$ norms are given by $||f||_{L_2}$ and $||f||_{L_\infty}$, respectively.

\section{Preliminary}
\label{section2}

Consider an asymptotically stable LQO system, which is formulated by the state-space description as
\begin{equation}\label{2.1.1}
    \Sigma : \left\{
    \begin{aligned}
        & \dot{x}(t) = Ax(t) + Bu(t), \quad x(0) = 0, \\
        & y(t) = C_o x(t) + \begin{pmatrix}
        x(t)^\top M_1 x(t) \\
        \dots \\
        x(t)^\top M_p x(t)
        \end{pmatrix},
    \end{aligned}
    \right.
\end{equation}
where $x(t) \in \mathbb{R}^n$, $u(t) \in \mathbb{R}^m$, and $y(t) \in \mathbb{R}^p$ are the state vector, input vector, and scalar output, respectively, defined for $t \in \left[0, t_{\mathrm{end}}\right]$.
The matrices $A\in\mathbb{S}^{n\times n}$, $B\in\mathbb{R}^{n\times m}$, $C_o\in\mathbb{R}^{p\times n}$, and $M_i\in\mathbb{R}^{n\times n}$ are constant matrices.
We assume that the system is of high order but with significantly fewer inputs than 
states, i.e., $m \ll n$.
The matrices $M_i$ are assumed to be symmetric, as for any $M_i \in \mathbb{R}^{n 
\times n}$, the quadratic term $x(t)^{\top}M_i x(t)$ can be replaced by $x(t)^{\top} 
\left(\frac{M_i +M_i^{\top}}{2}\right) x(t)$.
For simplicity of exposition, we concentrate on the following MISO system
\begin{equation}\label{2.1.2}
    \Sigma : \left\{
    \begin{aligned}
        & \dot{x}(t) = Ax(t) + Bu(t), \quad x(0) = 0, \\
        & y(t) = C x(t) + x(t)^\top M x(t),
    \end{aligned}
    \right.
\end{equation}
where $C\in\mathbb{R}^{1\times n}$, and $M\in\mathbb{R}^{n\times n}$.
However, all theoretical results and the proposed algorithm presented in this paper 
can be straightforwardly extended to \eqref{2.1.1} with some proper modifications.
\par

Our goal is to construct a reduced model (ROM) of order $r$ ($r \ll n$) that 
properly approximates the input-output behavior of the original system \eqref{2.1.2}.
The ROM is described by the following state-space representations
\begin{equation}\label{2.2}
	\hat{\Sigma}:\left\{\begin{array}{l}\dot{\hat{x}}(t)=\hat{A} \hat{x}(t)+\hat{B} u(t), \\\hat{y}(t)=\hat{C}\hat{x}(t)+\hat{x}(t)^{\top}\hat{M}\hat{x}(t),\end{array}\right.
\end{equation}
where $\hat{x}(t) \in \mathbb{R}^r$, $u(t) \in \mathbb{R}^m$, and $y(t) \in \mathbb{R}$ are the state vector, input vector, and scalar output, respectively, defined for $t \in \left[0, t_{\mathrm{end}}\right]$.
The ROM matrices $\hat{A}\in\mathbb{R}^{r\times r}$, $\hat{B}\in\mathbb{R}^{r\times 
m}$, $\hat{C}\in\mathbb{R}^{1\times r}$, and $\hat{M}\in\mathbb{R}^{r\times r}$ are 
constant matrices. Note that the stability and the quadratic output structure of 
(\ref{2.1.2}) will be ensured during the MOR procedure. 
\par

As detailed in \cite{Wang2026}, the controllability Gramian $P$ of the system \eqref{2.1.2} is defined as the unique, symmetric, positive semidefinite solution to the algebraic Lyapunov equation
\begin{equation}\label{2.3}
    AP+PA^\top+BB^\top=0.
\end{equation}
Correspondingly, the generalized observability Gramian $Q$ associated with the quadratic output is defined as the unique symmetric positive semidefinite solution to the generalized algebraic Lyapunov equation
\begin{equation}\label{2.4}
    A^\top Q+QA+C^\top C+MPM=0,
\end{equation}
where $P$ is the controllability Gramian from \eqref{2.3}.
The existence and uniqueness of these positive semidefinite solutions $P$ and $Q$ are guaranteed by the stability of the matrix $A$.
\par

The $H_2$-norm of the LQO system \eqref{2.1.2} is defined via its Volterra kernels as
\begin{equation}
    \|\Sigma\|_{H_2} = 
    \left(\int_0^\infty\|h_1(\sigma)\|_2^2\mathrm{d}\sigma+\int_0^\infty\int_0^\infty
    \|h_2(\sigma_1,\sigma_2)\|_2^2\mathrm{d}\sigma_1\mathrm{d}\sigma_2\right)^{\frac12},
\end{equation}
where $h_1(\sigma)=Ce^{A\sigma}B$ and 
$h_2(\sigma_{1},\sigma_{2})=\mathrm{vec}\left(B^{\top}e^{A^{\top}\sigma_{1}}Me^{A\sigma_{2}}B\right)^{\top}$
 are the linear and quadratic kernels, respectively.
Alternatively, this norm can be expressed algebraically in terms of the generalized observability Gramian $Q$.
As shown in \cite{Reiter2025}, the $H_2$-norm is given by
\begin{equation}\label{2.5}
    \|\Sigma\|_{H_{2}} = \sqrt{\mathrm{tr}\left(B^{\top}QB\right)}.
\end{equation}
\par

To evaluate the quality of the ROM $\hat{\Sigma}$, we consider the output error $y(t) - \hat{y}(t)$ when both the FOM $\Sigma$ and the ROM $\hat{\Sigma}$ are driven by the same input $u(t)$.
A key relationship, presented in \cite{Reiter2025}, bounds the $L_{\infty}$-norm of this time-domain error by the $H_2$-norm of the error system, $\Sigma_e := \Sigma-\hat{\Sigma}$.
The bound is given by
\begin{equation}
    \|y-\hat{y}\|_{L_\infty} := \sup_{t\geq0}|y(t)-\hat{y}(t)| \leq \|\Sigma_e\|_{H_2} (\|u\|_{L_2}+\|u\otimes u\|_{L_2}).
\end{equation}
This inequality provides a strong motivation for $H_2$-optimal model reduction.
It demonstrates that minimizing the $H_2$-norm of the error system provides a bound on the peak output error for any input $u(t)$ with finite $L_2$ and $L_2$-Kronecker norms.
\par

\section{$H_2$-optimal MOR problem}
\label{section3}
Building on the definition of the $H_2$-norm for LQO systems, the $H_2$-optimal MOR 
can be cast as an optimization problem.
As shown in \cite{Wang2026}, the $H_2$-optimal model reduction problem is formulated as the following optimization problem
\begin{equation}\label{3.1}
    \min_{\substack{\hat{A}\in\mathbb{S}^{r\times r},\,\hat{B}\in\mathbb{R}^{r\times m}\\\hat{C}\in\mathbb{R}^{1\times r},\,\hat{M}\in \text{Sym}_+(n)}}f(\hat{A},\hat{B},\hat{C},\hat{M}) = \|\Sigma-\hat{\Sigma}\|_{H_2}^2.
\end{equation}
With the constraint on the stability of ROM, the direct solution of this problem is 
challenging.
The primary difficulty stems from the highly non-convex nature of the stability constraint set, $\mathbb{S}^{r \times r}$ \cite{Sato2019}.
To overcome this challenge, we reformulate the problem as an equivalent, tractable optimization problem on a Riemannian manifold.
\par

For any asymptotically stable linear systems \eqref{2.1.2}, it can be 
equivalently expressed in the form
\begin{equation}\label{3.2}
    {\Sigma_2}:\left\{\begin{array}{l}
        \dot{x}(t)=(J-R)Qx(t)+Bu(t), \\
        y(t)=Cx(t)+x(t)^{\top}Mx(t),
    \end{array}\right.
\end{equation}
where $J$ and $R$ are defined based on $A$ and $Q$ as
\begin{gather}
    J := \frac{1}{2}(AQ^{-1} - Q^{-1}A^{\top}) \in \text{Skew}(n), \\
    R := -\frac{1}{2}(AQ^{-1} + Q^{-1}A^{\top}) \in \text{Sym}_+(n).
\end{gather}
Next, system \eqref{3.2} can be transformed into another equivalent form which is more suitable for reduction.
To achieve this, we first note that the matrix $Q \in \text{Sym}_+(n)$ from system \eqref{3.2} is SPD.
It therefore admits a unique Cholesky decomposition $Q = LL^{\top}$ 
\cite{Golub2012}, where $L \in \mathbb{R}^{n \times n}$ is a lower triangular matrix 
with positive diagonal entries.
We then apply the coordinate transformation $T = L^{\top}$, i.e., $\tilde{x}(t) = 
L^{\top} x(t)$.
This yields the equivalent system
\begin{equation}\label{3.3}
    {\Sigma_3}:\left\{\begin{array}{rcl}
        \dot{\tilde{x}}(t) & = & (\tilde{J} - \tilde{R})\tilde{x}(t) + \tilde{B}u(t), \\
        y(t) & = & \tilde{C}\tilde{x}(t) + \tilde{x}(t)^{\top}\tilde{M}\tilde{x}(t),
    \end{array}\right.
\end{equation}
where the transformed matrices are given by
\begin{align*}
    \tilde{J} &:= T J T^{-1} = (L^{\top}) J (L^{\top})^{-1} = L^{\top} J L, \\
    \tilde{R} &:= T R T^{-1} = (L^{\top}) R (L^{\top})^{-1} = L^{\top} R L, \\
    \tilde{B} &:= T B = L^{\top} B, \\
    \tilde{C} &:= C T^{-1} = C (L^{\top})^{-1} = C (L^{-1})^{\top}, \\
    \tilde{M} &:= T^{-{\top}} M T^{-1} = L^{-1} M (L^{-1})^{\top}.
\end{align*}
This transformation is a state-space isomorphism, so systems \eqref{3.2} and \eqref{3.3} are equivalent.
If we simply choose $Q$ as the identity matrix, then \eqref{2.1.2} takes the form
\begin{equation}\label{3.4}
    {\Sigma}:\left\{\begin{array}{l}
        \dot{x}(t)=(J-R)x(t)+Bu(t), \\
        y(t)=Cx(t)+x(t)^{\top}Mx(t),
    \end{array}\right.
\end{equation}
where $J = \text{skew}(A) \in \text{Skew}(n)$, $R = \text{sym}(A) \in \text{Sym}_+(n)$.
Therefore it is natural to construct the ROM of \eqref{2.1.2} in the following form
\begin{equation}\label{3.5}
	\hat{\Sigma}:\left\{\begin{array}{l}
    \dot{\hat{x}}(t)=(\hat{J}-\hat{R}) \hat{x}(t)+\hat{B} u(t), \\
    \hat{y}(t)=\hat{C}\hat{x}(t)+\hat{x}(t)^{\top}\hat{M}\hat{x}(t),
  \end{array}\right.
\end{equation}
where $\hat{x}(t) \in \mathbb{R}^r$, $u(t) \in \mathbb{R}^m$, and $y(t) \in \mathbb{R}$ are the state vector, input vector, and scalar output, respectively, defined for $t \in \left[0, t_{\mathrm{end}}\right]$.
The ROM matrices $\hat{J}\in\text{Skew}(r)$, $\hat{R}\in\text{Sym}_+(r)$, $\hat{B}\in\mathbb{R}^{r\times m}$, $\hat{C}\in\mathbb{R}^{1\times r}$, and $\hat{M}\in\mathbb{R}^{r\times r}$ are constant matrices.
For convenience, we use $\hat{A}$ to denote $\hat{J} - \hat{R}$.
\par

For this analysis, we construct the error system $\Sigma_{error} := \Sigma - \hat{\Sigma}$ between the FOM $\Sigma$ \eqref{2.1.2} and the ROM $\hat{\Sigma}$ \eqref{3.5}.
This error system $\Sigma_{error}$ possesses a state-space realization
\begin{equation}\label{3.6}
    \Sigma_{error}:\left\{\begin{array}{l}
        \dot{x}_e(t)=A_e x_e(t)+B_e u(t), \\
        y_e(t)=C_e x_e(t)+x_e(t)^{\top}M_e x_e(t),
    \end{array}\right.
\end{equation}
where $x_e(t) := \begin{pmatrix} x(t) \\ \hat{x}(t) \end{pmatrix}$, and the corresponding system matrices $(A_e, B_e, C_e, M_e)$ are defined as
\begin{equation}\label{3.7}
    (A_e, B_e, C_e, M_e) =
    \left(
    \begin{pmatrix}
        A & 0 \\
        0 & \hat{J}-\hat{R}
    \end{pmatrix},
    \begin{pmatrix}
        B \\ \hat{B}
    \end{pmatrix},
    \begin{pmatrix}
        C & -\hat{C}
    \end{pmatrix},
    \begin{pmatrix}
        M & 0 \\
        0 & -\hat{M}
    \end{pmatrix}
    \right),
\end{equation}
where $(A,B,C,M)$ and $(\hat{J},\hat{R},\hat{B},\hat{C},\hat{M})$ are the matrices for the FOM $\Sigma$ and ROM $\hat{\Sigma}$, respectively.
A direct computation using \eqref{3.6} and \eqref{3.7} confirms that the output $y_e(t)$ is indeed the difference between the original and reduced outputs: $y_e(t) = y(t) - \hat{y}(t)$.
Crucially, the error system $\Sigma_{error}$ defined in \eqref{3.6} is itself an LQO system, so its $H_2$-norm can be computed by applying \eqref{2.5}
\begin{equation}\label{3.8}
    \|\Sigma_{error}\|_{H_2} = \sqrt{\mathrm{tr}(B_e^\top Q_e B_e)}.
\end{equation}
Here, $Q_e$ is the generalized observability Gramian of $\Sigma_{error}$, satisfying
\begin{equation}
    A_e^\top Q_e + Q_e A_e + C_e^\top C_e + M_e P_e M_e = 0, \label{3.10}
\end{equation}
and $P_e$ is the controllability Gramian of $\Sigma_{error}$ satisfying
\begin{equation}
    A_e P_e + P_e A_e^\top + B_e B_e^\top = 0. \label{3.9}
\end{equation}
To further analyze the $H_2$-norm of the error system, we partition the error Gramians $P_e$ and $Q_e$ according to the structure of $A_e$
\begin{equation}\label{3.11}
    P_e = \begin{pmatrix} P & X \\ X^\top & \hat{P} \end{pmatrix}, \quad Q_e = \begin{pmatrix} Q & Y \\ Y^\top & \hat{Q} \end{pmatrix}.
\end{equation}
By substituting the augmented system matrices \eqref{3.7} and the partitions \eqref{3.11} into the Lyapunov equations \eqref{3.9} and \eqref{3.10}, the system of equations decouples.
The top-left blocks, $P$ and $Q$, are precisely the controllability and generalized observability Gramians of FOM $\Sigma$, satisfying
\begin{align}
    AP + PA^\top + BB^\top &= 0, \label{3.12} \\
    A^\top Q + QA + C^\top C + MPM &= 0. \label{3.13}
\end{align}
The bottom-right blocks, $\hat{P}$ and $\hat{Q}$, are the Gramians of the ROM $\hat{\Sigma}$.
Specifically, $\hat{P}$ is the ROM controllability Gramian, and $\hat{Q}$ is the ROM generalized observability Gramian, satisfying
\begin{align}
    (\hat{J}-\hat{R})\hat{P} + \hat{P}(\hat{J}-\hat{R})^\top + \hat{B}\hat{B}^\top &= 0, \label{3.14} \\
    (\hat{J}-\hat{R})^\top \hat{Q} + \hat{Q}(\hat{J}-\hat{R}) + \hat{C}^\top \hat{C} + \hat{M}\hat{P}\hat{M} &= 0. \label{3.15}
\end{align}
The off-diagonal blocks $X$ and $Y$ are the solutions to the following cross-coupling Sylvester equations
\begin{align}
    AX + X(\hat{J}-\hat{R})^\top + B\hat{B}^\top &= 0, \label{3.16} \\
    A^\top Y + Y(\hat{J}-\hat{R}) - C^\top \hat{C} - MX\hat{M} &= 0. \label{3.17}
\end{align}
With these sub-blocks, we can expand the $H_2$-norm expression from \eqref{3.8}.
Substituting $B_e$ and $Q_e$ yields
\begin{align*}
    \|\Sigma_{error}\|_{H_2}^2 &= \mathrm{tr}(B_e^\top Q_e B_e) \\
    &= \mathrm{tr}\left( \begin{pmatrix} B \\ \hat{B} \end{pmatrix}^\top \begin{pmatrix} Q & Y \\ Y^\top & \hat{Q} \end{pmatrix} \begin{pmatrix} B \\ \hat{B} \end{pmatrix} \right) \\
    &= \mathrm{tr}\left( \begin{pmatrix} B^\top & \hat{B}^\top \end{pmatrix} \begin{pmatrix} QB + Y\hat{B} \\ Y^\top B + \hat{Q}\hat{B} \end{pmatrix} \right) \\
    &= \mathrm{tr}\left( B^\top(QB + Y\hat{B}) + \hat{B}^\top(Y^\top B + \hat{Q}\hat{B}) \right) \\
    &= \mathrm{tr}(B^\top Q B + B^\top Y \hat{B} + \hat{B}^\top Y^\top B + \hat{B}^\top \hat{Q} \hat{B}).
\end{align*}
Using the cyclic property of the trace ($\mathrm{tr}(\hat{B}^\top Y^\top B) = \mathrm{tr}(B^\top Y \hat{B})$), this simplifies our cost function $f$.
The optimization problem is to minimize the squared $H_2$-norm error, defined as
\begin{equation}\label{3.18}
    f(\hat{J}, \hat{R}, \hat{B}, \hat{C}, \hat{M}) = \|\Sigma_{error}\|_{H_2}^2 = \mathrm{tr}(B^\top Q B + 2B^\top Y \hat{B} + \hat{B}^\top \hat{Q} \hat{B}),
\end{equation}
where the matrices $Q, Y, \hat{Q}$ depend on the optimization variables through the equations \eqref{3.13}, \eqref{3.17}, and \eqref{3.15}.
The $H_2$-optimal model reduction problem is thus formulated as the following constrained optimization problem
\begin{equation}\label{3.19}
    \min_{(\hat{J},\hat{R},\hat{B},\hat{C},\hat{M}) \in \mathcal{M}}f(\hat{J},\hat{R},\hat{B},\hat{C},\hat{M}).
\end{equation}
where the product manifold $\mathcal{M}$ is defined as $\mathcal{M} := \text{Skew}(r) \times \text{Sym}_+(r) \times \mathbb{R}^{r \times m} \times \mathbb{R}^{1 \times r} \times \text{Sym}(r)$.

\section{Riemannian optimization on the product manifold}
\label{section4}

In this section, we first extract the explicit expression for Riemannian gradient of 
(\ref{3.19}) based on the geometric properties of the product manifold 
$\mathcal{M}$, and then develop an optimization algorithm on the product manifold to 
produce stable reduced models iteratively.

\subsection{Riemannian geometry of the product manifold}
\label{section4.1}

Formally, a manifold is a couple $(\mathcal{X}, \mathcal{A}^+)$, where $\mathcal{X}$ 
is a set and $\mathcal{A}^+$ is a maximal atlas of $\mathcal{X}$ inducing a 
second-countable Hausdorff topology \cite{Absil2008,Boumal2023}.
The manifold $\mathcal{M}$ in our work is a product manifold, composed of the vector space of skew-symmetric matrices, the manifold of SPD matrices, two Euclidean spaces, and the vector space of symmetric matrices.
We now recall the geometric properties of these constituent manifolds, which 
facilitate the analysis on the product manifold $\mathcal{M}$ a lot.
\par

The set of skew-symmetric matrices $\text{Skew}(n)$ is a linear subspace of $\mathbb{R}^{n \times n}$.
It can be regarded as a flat Riemannian manifold, whose geometry is inherited from the ambient Euclidean space $\mathbb{R}^{n \times n}$.

\paragraph{Tangent Space and Riemannian Metric}
For any $X \in \text{Skew}(n)$, the tangent space $T_X\text{Skew}(n)$ is canonically identified with the subspace $\text{Skew}(n)$ itself, that is, $T_X\text{Skew}(n) \simeq \text{Skew}(n)$.
The Riemannian metric is the restriction of the standard Frobenius inner product
$$
\langle \xi, \eta \rangle_X := \text{tr}(\xi^\top \eta), \quad \text{for any } \xi, \eta \in T_X\text{Skew}(n).
$$

\paragraph{Riemannian Gradient}
The orthogonal projection from $\mathbb{R}^{n \times n}$ onto $T_X\text{Skew}(n)$ is given by the operator $\operatorname{skew}(\cdot)$.
Let $f: \text{Skew}(n) \to \mathbb{R}$ be a smooth function and $\bar{f}$ be its smooth extension to $\mathbb{R}^{n \times n}$.
The Riemannian gradient $\mathrm{grad}\,f(X)$ is the projection of the Euclidean gradient $\nabla \bar{f}(X)$ onto the tangent space
\begin{equation}
    \mathrm{grad}\,f(X) = \text{skew}(\nabla \bar{f}(X)).
\end{equation}

\paragraph{Retraction and Vector Transport}
Given the flat geometry of the manifold, the exponential map $\text{Exp}_X(\xi)$ coincides with the simplest possible retraction, which is vector addition
$$
R_X(\xi) = \text{Exp}_X(\xi) = X + \xi, \quad \text{for } \xi \in T_X\text{Skew}(n).
$$
Similarly, the vector transport $\mathcal{T}_\eta(\xi)$ is simply the identity map, as all tangent spaces are identical
$$
\mathcal{T}_\eta(\xi) = \xi.
$$

The set of symmetric matrices, denoted by $\text{Sym}(n)$, is also a linear subspace of $\mathbb{R}^{n \times n}$, so the Riemannian geometry is similarly inherited from the ambient Euclidean space $\mathbb{R}^{n \times n}$.
The set of SPD matrices, denoted by $\text{Sym}_+(n)$, is an open cone within the vector space of $\text{Sym}(n)$.
It is well-known that $\text{Sym}_+(n)$ is a Riemannian manifold \cite{Moakher2005,Vandereycken2009}.

\paragraph{Tangent Space and Riemannian Metric}
The tangent space at any point $X \in \text{Sym}_+(n)$, denoted $T_X\text{Sym}_+(n)$, is canonically identified with $\text{Sym}(n)$, that is, $T_X\text{Sym}_+(n) \simeq \text{Skew}(n)$.
We equip $\text{Sym}_+(n)$ with the Riemannian metric \cite{Jeuris2012}
\begin{equation}
    \langle \xi, \eta \rangle_X := \text{tr}(X^{-1}\xi X^{-1}\eta), \quad \text{for any } \xi, \eta \in T_X\text{Sym}_+(n).
\end{equation}

\paragraph{Riemannian Gradient}
The orthogonal projection from $\mathbb{R}^{n \times n}$ onto $T_X\text{Sym}_+(n)$ is given by the operator $\operatorname{sym}(\cdot)$.
Let $f: \text{Sym}_+(n) \to \mathbb{R}$ be a smooth function with a smooth extension $\bar{f}$ to $\mathbb{R}^{n \times n}$.
The Riemannian gradient $\mathrm{grad}\,f(X)$ with respect to the metric above is given by \cite{Obara2024}
\begin{equation}
    \mathrm{grad}\,f(X) = X \, \text{sym}(\nabla \bar{f}(X)) \, X,
\end{equation}
where $\nabla \bar{f}(X)$ is the Euclidean gradient of $\bar{f}$ at $X$.

\paragraph{Retraction and Vector Transport}
We utilize the exponential map as the retraction and its corresponding parallel transport. 
For the SPD manifold endowed with the affine-invariant metric, the exponential map at $X \in \text{Sym}_+(n)$ for a tangent vector $\eta \in T_X\text{Sym}_+(n)$ is given by
\begin{equation}
    R_X(\eta) = \text{Exp}_X(\eta) = X^{\frac{1}{2}} \exp\left(X^{-\frac{1}{2}} \eta X^{-\frac{1}{2}}\right) X^{\frac{1}{2}}.
\end{equation}

To transport a tangent vector $\xi \in T_X\text{Sym}_+(n)$ along the geodesic defined by $\eta$ to the new tangent space at $R_X(\eta)$, we employ the parallel transport.
The parallel transport is isometric, which preserves the Riemannian inner product of two transported vectors.
The closed-form expression of the parallel transport on $\text{Sym}_+(n)$ is given by
\begin{equation}
    \mathcal{T}_\eta(\xi) = E \xi E^T,
\end{equation}
where the transformation matrix $E$ is defined as
\begin{equation}
    E = X^{\frac{1}{2}} \exp\left(\frac{1}{2} X^{-\frac{1}{2}} \eta X^{-\frac{1}{2}}\right) X^{-\frac{1}{2}}.
\end{equation}

Using the geometric tools established for the constituent manifolds, we now define the geometry of the product manifold $\mathcal{M}$.
Let the manifold $\mathcal{M}$ be the product of the spaces for the optimization variables:
$$
\mathcal{M} := \text{Skew}(n) \times \text{Sym}_+(n) \times \mathbb{R}^{n \times m} \times \mathbb{R}^{1 \times n} \times \text{Sym}(n).
$$

\paragraph{Tangent Space and Riemannian Metric}
Let $x = (\hat{J}, \hat{R}, \hat{B}, \hat{C}, \hat{M}) \in \mathcal{M}$ denote a point on $\mathcal{M}$.
The tangent space $T_x\mathcal{M}$ is the product of the tangent spaces of each component
$$
T_x\mathcal{M} = \text{Skew}(n) \times \text{Sym}(n) \times \mathbb{R}^{n \times m} \times \mathbb{R}^{1 \times n} \times \text{Sym}(n).
$$
A tangent vector at $x$ is denoted by $\xi = (\xi_J, \xi_R, \xi_B, \xi_C, \xi_M) \in T_x\mathcal{M}$.
The Riemannian metric $\langle \cdot, \cdot \rangle_x$ at $x \in \mathcal{M}$ is defined as the sum of the metrics on the constituent manifolds.
For two tangent vectors $\xi, \eta \in T_x\mathcal{M}$
\begin{align}
    \langle \xi, \eta \rangle_x = & \mathrm{tr}(\xi_J^\top \eta_J) + \mathrm{tr}(\hat{R}^{-1}\xi_R \hat{R}^{-1}\eta_R) + \mathrm{tr}(\xi_B^\top \eta_B) + \mathrm{tr}(\xi_C \eta_C^\top) + \mathrm{tr}(\xi_M^\top \eta_M).
\end{align}

\paragraph{Riemannian Gradient}
Let $f: \mathcal{M} \to \mathbb{R}$ be a smooth cost function, and let $\bar{f}$ be its smooth extension to the ambient Euclidean space.
Let $\nabla \bar{f}(x) = (\nabla_{\hat{J}} \bar{f}, \nabla_{\hat{R}} \bar{f}, \nabla_{\hat{B}} \bar{f}, \nabla_{\hat{C}} \bar{f}, \nabla_{\hat{M}} \bar{f})$ be the Euclidean gradient of $\bar{f}$.
Here, we define the matrix $\nabla_{\hat{J}} \bar{f}$ as the partial derivative of $\bar{f}$ with respect to the matrix variable $\hat{J}$.
The $(i,j)$-th element of this matrix is given by $(\nabla_{\hat{J}} \bar{f})_{ij} = \frac{\partial \bar{f}}{\partial (\hat{J})_{ij}}$.
Other matrices are defined similarly.
The Riemannian gradient $\mathrm{grad} f(x)$ is the component-wise projection of $\nabla \bar{f}(x)$ onto $T_x\mathcal{M}$
$$
\mathrm{grad} f(x) = \left( \text{skew}(\nabla_{\hat{J}} \bar{f}), \quad \hat{R} \, 
\text{sym}(\nabla_{\hat{R}} \bar{f}) \, \hat{R}, \quad \nabla_{\hat{B}} \bar{f}, 
\quad \nabla_{\hat{C}} \bar{f}, \quad \text{sym}(\nabla_{\hat{M}} \bar{f}) \right).
$$

\paragraph{Retraction and Vector Transport}
The retraction $R_x(\xi)$ and vector transport $\mathcal{T}_{\eta}(\xi)$ are also defined component-wise
$$
R_x(\xi) = \left( \hat{J} + \xi_J, \quad \hat{R}^{\frac{1}{2}} \exp\left(\hat{R}^{-\frac{1}{2}} \xi_R \hat{R}^{-\frac{1}{2}}\right) \hat{R}^{\frac{1}{2}}, \quad \hat{B} + \xi_B, \quad \hat{C} + \xi_C, \quad \hat{M} + \xi_M \right),
$$
and the strictly isometric vector transport is defined as
$$
\mathcal{T}_{\eta}(\xi) = \left( \xi_J, \quad E_{\eta_R} \xi_R E_{\eta_R}^T, \quad \xi_B, \quad \xi_C, \quad \xi_M \right),
$$
where $E_{\eta_R}$ is given by
$$
E_{\eta_R} = \hat{R}^{\frac{1}{2}} \exp\left(\frac{1}{2} \hat{R}^{-\frac{1}{2}} \eta_R \hat{R}^{-\frac{1}{2}}\right) \hat{R}^{-\frac{1}{2}}.
$$
This provides the complete set of geometric tools required for optimization on $\mathcal{M}$.

\subsection{Riemannian optimal reduction based on the product manifold}
\label{section4.2}

Based on the geometric framework established in the previous subsection, we first 
derive the explicit expression for Riemannian gradient, 
and then present an iterative MOR algorithm based on the Riemannian BFGS method. 
A proposition is introduced first to reveal the trace properties of Sylvester 
equations.

\newtheorem{proposition}{Proposition}
\begin{proposition}
\label{prop1}
If $P$ and $Q$ satisfy $AP+PB+X=0$ and $A^\top Q+QB^\top+Y=0$, then it holds that $\mathrm{tr}(Y^\top P)=\mathrm{tr}(X^\top Q)$.
\end{proposition}
With the aid of \cref{prop1}, the following theorem provides the Riemannian gradient of the cost function $f$ defined in \eqref{3.18}.
\par

\newtheorem{theorem}{Theorem}
\begin{theorem}
\label{theorem1}
Consider the asymptotically stable systems (\ref{2.1.2}) and (\ref{2.2}).
The Riemannian gradient $\mathrm{grad}f$ of the cost function on the product manifold $\mathcal{M}$ is
$$
2 \left( \mathrm{skew}(K^\top X + L\hat{P}), \; -\hat{R} \mathrm{sym}(K^\top X + 
L\hat{P})\hat{R}, \; K^\top B + L\hat{B}, \; C\hat{P} - CX, \; \hat{P}\hat{M}\hat{P} 
- X^\top M X \right),
$$
where $X, \hat{P}$ are the solutions to \eqref{3.16}, \eqref{3.14}, and $K, L$ are the solutions to the following Sylvester equations
\begin{align}
  A^\top K+K(\hat{J}-\hat{R})-C^\top\hat{C}-2MX\hat{M}=&0,\label{4.2.1} \\
	(\hat{J}-\hat{R})^\top L+L(\hat{J}-\hat{R})+\hat{C}^\top\hat{C}+2\hat{M}\hat{P}\hat{M}=&0.\label{4.2.2}
\end{align}
\end{theorem}

\begin{proof}
Let $x = (\hat{J}, \hat{R}, \hat{B}, \hat{C}, \hat{M}) \in \mathcal{M}$ be the current point, and $\xi = (\xi_J, \xi_R, \xi_B, \xi_C, \xi_M) \in T_x\mathcal{M}$ be a tangent vector.
Let $\bar{f}$ be the smooth extension to the ambient Euclidean space of $f$.
We compute the directional derivative $D\bar{f}(x)[\xi]$ by differentiating the cost function \eqref{3.18}
\begin{equation}\label{4.2.3}
    D\bar{f}(x)[\xi] = \mathrm{tr}\left( 2B^\top Y \xi_B + 2\hat{B}^\top \dot{Y}^\top B + 2\hat{B}^\top \hat{Q} \xi_B + \hat{B}^\top \dot{\hat{Q}} \hat{B} \right),
\end{equation}
where $\dot{Y} = DY(x)[\xi]$ and $\dot{\hat{Q}} = D\hat{Q}(x)[\xi]$ are the directional derivatives satisfying the linearized Sylvester and Lyapunov equations
\begin{align}
    A^\top \dot{Y} + \dot{Y}(\hat{J}-\hat{R}) &= \mathcal{N}_Y + M\dot{X}\hat{M}, \label{4.2.4} \\
    (\hat{J}-\hat{R})^\top \dot{\hat{Q}} + \dot{\hat{Q}}(\hat{J}-\hat{R}) &= \mathcal{N}_Q - \hat{M}\dot{\hat{P}}\hat{M}, \label{4.2.5}
\end{align}
where $\mathcal{N}_Y$ and $\mathcal{N}_Q$ collect the terms explicitly dependent on $\xi$
\begin{align*}
    \mathcal{N}_Y &= -Y(\xi_J - \xi_R) + C^\top \xi_C + M X \xi_M, \\
    \mathcal{N}_Q &= -[(\xi_J - \xi_R)^\top \hat{Q} + \hat{Q}(\xi_J - \xi_R) + \hat{C}^\top \xi_C + \xi_C^\top \hat{C} + \xi_M \hat{P} \hat{M} + \hat{M} \hat{P} \xi_M].
\end{align*}
We proceed by analyzing the trace terms involving $\dot{Y}$ and $\dot{\hat{Q}}$ separately, which motivates the construction of the adjoint variables $K$ and $L$.

\textbf{Step 1: Analysis of the $\dot{Y}$ term.}
Using the cyclic property of the trace and \cref{prop1} and substituting \eqref{4.2.4}, the term involving $\dot{Y}$ in \eqref{4.2.3} can be expanded as
\begin{equation}\label{4.2.6}
    2\mathrm{tr}(\hat{B}^\top \dot{Y}^\top B) = 2\mathrm{tr}(\hat{B} B^\top \dot{Y}) = - 2\mathrm{tr}(X^\top \mathcal{N}_Y) - 2\mathrm{tr}(X^\top M \dot{X} \hat{M}).
\end{equation}
The last term, $\mathrm{tr}(X^\top M \dot{X} \hat{M})$, couples the derivative $\dot{X} = DX(x)[\xi]$ with other matrices.
To handle this, an auxiliary adjoint variable $S_1$ satisfying the following Lyapunov equation is introduced
\begin{equation}\label{4.2.7}
    A^\top S_1 + S_1 (\hat{J}-\hat{R}) + M X \hat{M} = 0.
\end{equation}
Applying \cref{prop1} to the equation for $\dot{X}$ (linearization of \eqref{3.16}) and the equation for $S_1$ \eqref{4.2.7}, we obtain the identity
\begin{equation}
    \mathrm{tr}(X^\top M \dot{X} \hat{M}) = \mathrm{tr}((M X \hat{M})^\top \dot{X}) = \mathrm{tr}\left( S_1^\top [ X(\xi_J - \xi_R)^\top + B \xi_B^\top ] \right).
\end{equation}
Substituting this back into \eqref{4.2.6}, we can combine $Y$ and $S_1$.
Defining $K := Y - S_1$, and noting that subtracting \eqref{4.2.7} from \eqref{3.17} yields the defining equation \eqref{4.2.1} for $K$, we derive
\begin{equation}\label{4.2.8}
    2\mathrm{tr}(\hat{B}^\top \dot{Y}^\top B) = 2\mathrm{tr}\left( (\xi_J - \xi_R)^\top K^\top X -\xi_B^\top S_1^\top B - \xi_C^\top C X - \xi_M X^\top M X \right).
\end{equation}

\textbf{Step 2: Analysis of the $\dot{\hat{Q}}$ term.}
Similarly, for the term involving $\dot{\hat{Q}}$, we obtain
\begin{equation}\label{4.2.9}
    \mathrm{tr}(\hat{B}^\top \dot{\hat{Q}} \hat{B}) = \mathrm{tr}(\hat{B} \hat{B}^\top \dot{\hat{Q}}) = - \mathrm{tr}(\hat{P}^\top \mathcal{N}_Q) + \mathrm{tr}(\hat{P}^\top \hat{M} \dot{\hat{P}} \hat{M}).
\end{equation}
To eliminate the dependence on $\dot{\hat{P}} = D\hat{P}(x)[\xi]$, we introduce a second auxiliary variable $S_2$ satisfying
\begin{equation}\label{4.2.10}
    (\hat{J}-\hat{R})^\top S_2 + S_2 (\hat{J}-\hat{R}) + \hat{M}\hat{P}\hat{M} = 0.
\end{equation}
Again, applying \cref{prop1} to the equations for $\dot{\hat{P}}$ and $S_2$ yields
\begin{equation}
    \mathrm{tr}(\hat{P}^\top \hat{M} \dot{\hat{P}} \hat{M}) = \mathrm{tr}(S_2^\top [(\xi_J - \xi_R)\hat{P} + \hat{P}(\xi_J - \xi_R)^\top + \xi_B \hat{B}^\top + \hat{B} \xi_B^\top]) = 2\mathrm{tr}\left((\xi_J - \xi_R)^\top S_2 \hat{P} + \xi_B^\top S_2 \hat{B} \right).
\end{equation}
We define $L := \hat{Q} + S_2$. Adding \eqref{4.2.10} to \eqref{3.15} results in the defining equation \eqref{4.2.2} for $L$.
Substituting these results into \eqref{4.2.9} allows us to express the contribution purely in terms of $L$
\begin{equation}\label{4.2.11}
    \mathrm{tr}(\hat{B}^\top \dot{\hat{Q}} \hat{B}) = 2\mathrm{tr}\left( (\xi_J - \xi_R)^\top L \hat{P} + \xi_B^\top S_2 \hat{B} + \xi_C^\top \hat{C} \hat{P} + \xi_M^\top \hat{P} \hat{M} \hat{P} \right).
\end{equation}

\textbf{Step 3: Synthesis of the Gradient.}
Finally, combining \eqref{4.2.8} and \eqref{4.2.11} into the original differential \eqref{4.2.3}, and grouping terms with respect to the variation components $(\xi_J, \xi_R, \xi_B, \xi_C, \xi_M)$, we obtain the Euclidean gradient $\nabla \bar{f}$
\begin{align*}
    D\bar{f}(x)[\xi] &= 2\mathrm{tr}\left( (\xi_J - \xi_R)^\top (K^\top X + L \hat{P}) \right) + 2\mathrm{tr}\left( \xi_B^\top (K^\top B + L \hat{B}) \right) \\
    &+ 2\mathrm{tr}\left( \xi_C (\hat{C} \hat{P} - C X)^\top \right) + 2\mathrm{tr}\left( \xi_M^\top (\hat{P} \hat{M} \hat{P} - X^\top M X) \right) = \langle \nabla \bar{f}(x), \xi \rangle_{\text{Euclidean}},
\end{align*}
where the components of the Euclidean gradient $\nabla \bar{f}(x) = (\nabla_{\hat{J}}\bar{f}, \nabla_{\hat{R}}\bar{f}, \nabla_{\hat{B}}\bar{f}, \nabla_{\hat{C}}\bar{f}, \nabla_{\hat{M}}\bar{f})$ are identified as
\begin{align*}
    \nabla_{\hat{J}}\bar{f} &= 2(K^\top X + L \hat{P}), \\
    \nabla_{\hat{R}}\bar{f} &= -2(K^\top X + L \hat{P}), \\
    \nabla_{\hat{B}}\bar{f} &= 2(K^\top B + L \hat{B}), \\
    \nabla_{\hat{C}}\bar{f} &= 2(\hat{C} \hat{P} - C X), \\
    \nabla_{\hat{M}}\bar{f} &= 2(\hat{P} \hat{M} \hat{P} - X^\top M X).
\end{align*}
The Riemannian gradient is obtained by projecting the Euclidean gradients onto the 
tangent spaces $\text{Skew}(r)$, $\text{Sym}(r)$, and $\text{Sym}_+(r)$ as described 
in Section \ref{section4.1}. It concludes the proof.
\end{proof}

Now we are in a position to solve the optimization problem \eqref{3.19} by a 
Limited-memory Riemannian BFGS (LRBFGS) method \cite{Huang2018}.
While the standard Riemannian BFGS method constructs a dense approximation of the Hessian matrix, it becomes computationally expensive for large-scale problems.
The LRBFGS method circumvents this by implicitly approximating the Hessian using a set of recent step and gradient variations, significantly reducing computational costs.
\par
Let $x_k \in \mathcal{M}$ and $\eta_k \in T_{x_k}\mathcal{M}$ denote the current iterate and the search direction at step $k$, respectively. The search direction is computed as
\begin{equation}
    \eta_k = -\mathcal{H}_k(\mathrm{grad}f(x_k)), \label{eq:search_dir}
\end{equation}
where $\mathcal{H}_k$ is the inverse Hessian approximation operator.
In LRBFGS, $\eta_k$ is computed via a Riemannian adaptation of the two-loop recursion using the most recent $m$ stored history vectors. For more details of the two-loop recursion, see \cite{Huang2018}.
The subsequent iterate $x_{k+1}$ is generated via the retraction map
\begin{equation}
    x_{k+1} = \mathcal{R}_{x_k}(t_k \eta_k), \label{eq:retraction_update}
\end{equation}
where $t_k > 0$ is the step length, determined by a backtracking line search satisfying the Armijo condition
\begin{equation}
    f(\mathcal{R}_{x_k}(t_k \eta_k)) \le f(x_k) + c_1 t_k \langle \mathrm{grad}f(x_k), \eta_k \rangle_{x_k}, \label{eq:armijo}
\end{equation}
with a constant $0 < c_1 < 1$.
To update the Hessian approximation, $s_k \in T_{x_{k+1}}\mathcal{M}$ and $y_k \in T_{x_{k+1}}\mathcal{M}$ are defined as
\begin{align}
    s_k &= \mathcal{T}_{t_k \eta_k}(t_k \eta_k), \label{eq:sk} \\
    y_k &= \mathrm{grad}f(x_{k+1}) - \mathcal{T}_{t_k \eta_k}(\mathrm{grad}f(x_k)), \label{eq:yk}
\end{align}
where $\mathcal{T}$ is the isometric vector transport.

To guarantee global convergence for non-convex functions and to preserve the positive definiteness of the Hessian approximation without the Wolfe conditions, we employ the cautious update rule
\begin{equation}
    \frac{\langle y_k, s_k \rangle_{x_{k+1}}}{\|s_k\|_{x_{k+1}}^2} \ge \vartheta(\|\mathrm{grad}f(x_k)\|_{x_k}), \label{eq:cautious}
\end{equation}
where $\vartheta(t) = c_{cautious} t$ is a strictly increasing threshold function with $c_{cautious} > 0$.
If the condition \eqref{eq:cautious} is satisfied, the pair $(s_k, y_k)$ is accepted into the memory buffer.
Because the tangent spaces change at each iteration, all previously stored history vectors $(s_i, y_i)$ in the memory buffer must be transported from $T_{x_k}\mathcal{M}$ to $T_{x_{k+1}}\mathcal{M}$ at every step.

The main steps of LRBFGS optimization procedure on the product manifold 
$\mathcal{M}$ is summarized in Algorithm \ref{alg1}.

\begin{algorithm}
\caption{$H_2$-optimal MOR of LQO systems via Riemannian BFGS Method on product 
manifold}
\label{alg1}
\begin{algorithmic}[1]
\REQUIRE The coefficient matrices of LQO systems, memory size $m$, tolerance $\epsilon > 0$.
\ENSURE ROM determined by the matrices $\hat{J}, \hat{R}, \hat{B}, \hat{C}, \hat{M}$.

\STATE \textbf{Initialize:} Choose an initial point $x_0 = (\hat{J}_0, \hat{R}_0, \hat{B}_0, \hat{C}_0, \hat{M}_0) \in \mathcal{M}$.
\STATE Compute the Riemannian gradient $\mathrm{grad}f(x_0)$ via \cref{theorem1}.
\FOR{$k = 0, 1, \dots$ (until convergence)}
    \STATE Compute search direction $\eta_k \gets -\mathcal{H}_k(\mathrm{grad}f(x_k))$ using the two-loop recursion.
    \STATE Choose a step length $t_k$ satisfying the Armijo condition \eqref{eq:armijo} via backtracking.
    \STATE Update the point $x_{k+1} \gets \mathcal{R}_{x_k}(t_k \eta_k)$.
    \STATE Calculate the new gradient $\mathrm{grad}f(x_{k+1})$.
    \STATE Compute the step variation $s_k$ and gradient variation $y_k$ via \eqref{eq:sk} and \eqref{eq:yk}.
    \STATE Transport all stored history vectors $s_i, y_i$ from $T_{x_k}\mathcal{M}$ to $T_{x_{k+1}}\mathcal{M}$ using $\mathcal{T}$.
    \IF{the cautious condition \eqref{eq:cautious} is satisfied}
        \STATE Add $(s_k, y_k)$ into the memory buffers and discard the oldest pair if the size exceeds $m$.
    \ENDIF
\ENDFOR

\STATE \textbf{Return:} $\hat{\Sigma} = (\hat{J}, \hat{R}, \hat{B}, \hat{C}, \hat{M})$.
\end{algorithmic}
\end{algorithm}

{\bf Remark 1} Note that throughout the proof of Theorem \ref{theorem1}, the 
variables $\hat{J}$ and $\hat{R}$ consistently appear as a coupled term $\hat{J} - 
\hat{R}$.
As a result, the Euclidean gradients of $\bar{f}$ with respect to these two variables are identical up to a sign difference.
Nevertheless, decoupling $\hat{A}$ into $\hat{J} - \hat{R}$ becomes justified at the manifold level.
Specifically, the set of skew-symmetric matrices $\mathrm{Skew}(r)$ is intrinsically a linear subspace of the Euclidean space $\mathbb{R}^{r \times r}$, whereas the set of SPD matrices $\mathrm{Sym}_+(r)$ constitutes a highly non-linear Riemannian manifold equipped with entirely different geometric properties.
Although their Euclidean gradient counterparts differ only by a sign, the Riemannian gradient of $f$ with respect to $\hat{J}$ and $\hat{R}$ take substantially different forms, thereby inducing distinct update behaviors during the iterations.

\section{Numerical results}
\label{section5}
This example comes from a one-dimensional advection-diffusion equation, given in 
\cite{Diaz2023,Reiter2025}.
The system dynamics are governed by the following partial differential equation
$$
\frac{\partial v(t, x)}{\partial t} - \alpha \frac{\partial^2 v(t, x)}{\partial x^2} + \beta \frac{\partial v(t, x)}{\partial x} = 0, \quad x \in (0, 1), \ t > 0,
$$
subject to the homogeneous initial condition $v(0, x) = 0$ for $x \in [0, 1]$, and the time-dependent boundary conditions
$$
v(t, 0) = u_0(t), \quad \alpha \frac{\partial v(t, x)}{\partial x}\bigg|_{x=1} = u_1(t).
$$
Capturing the spatial variance of the state's deviation from a predefined reference naturally leads to a quadratic output equation
$$
y(t) = \frac{1}{2} \int_0^1 \left( v(t, x) - 1 \right)^2 dx.
$$
Spatial discretization of the PDE is performed using a finite difference scheme over a grid of $n=300$ points.
The diffusion and advection coefficients are set to $\alpha = 0.01$ and $\beta = 1$, respectively.
This discretization yields a continuous-time MISO LQO system of order $n=300$, 
driven by $m=2$ inputs defined by $u(t) = [u_0(t), u_1(t)]^{\top}$.
\par

The reduced order is $r=10$ in the simulation.
We employ the BT procedure given in \cite{Benner2021} to get the 
initial point $x_0 
\in \mathcal{M}$ for the iteration.
The memory size for the two-loop recursion is restricted to $m = 10$. 
For the backtracking line search based on the Armijo condition, the sufficient decrease constant and the step-size decay factor are chosen as $c_1 = 10^{-4}$ and $\delta = 0.5$, respectively. 
Furthermore, the coefficient for the cautious update threshold is set to $c_{\text{cautious}} = 10^{-4}$. 
The algorithm is programmed to terminate when the relative Riemannian gradient norm falls below $10^{-2}$, or when the change in the cowt function between two consecutive iterations drops below $10^{-8}$.
\par

\begin{figure}[ht]
    \centering
    \begin{minipage}{0.45\textwidth}
    	\centering
        \includegraphics[width=\textwidth]{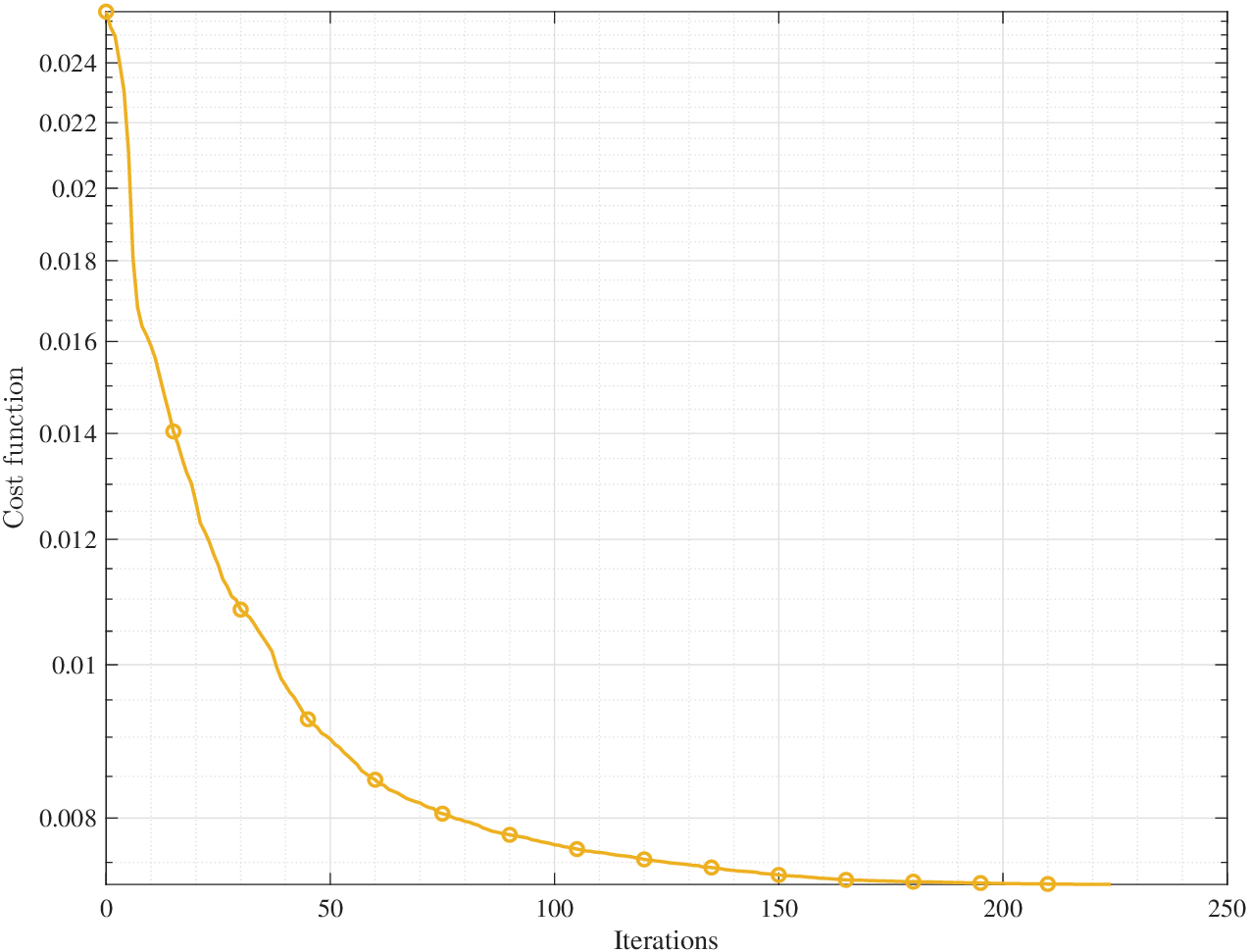}
        \caption*{(a)}
    \end{minipage}
    \hfill
    \begin{minipage}{0.45\textwidth}
    	\centering
        \includegraphics[width=\textwidth]{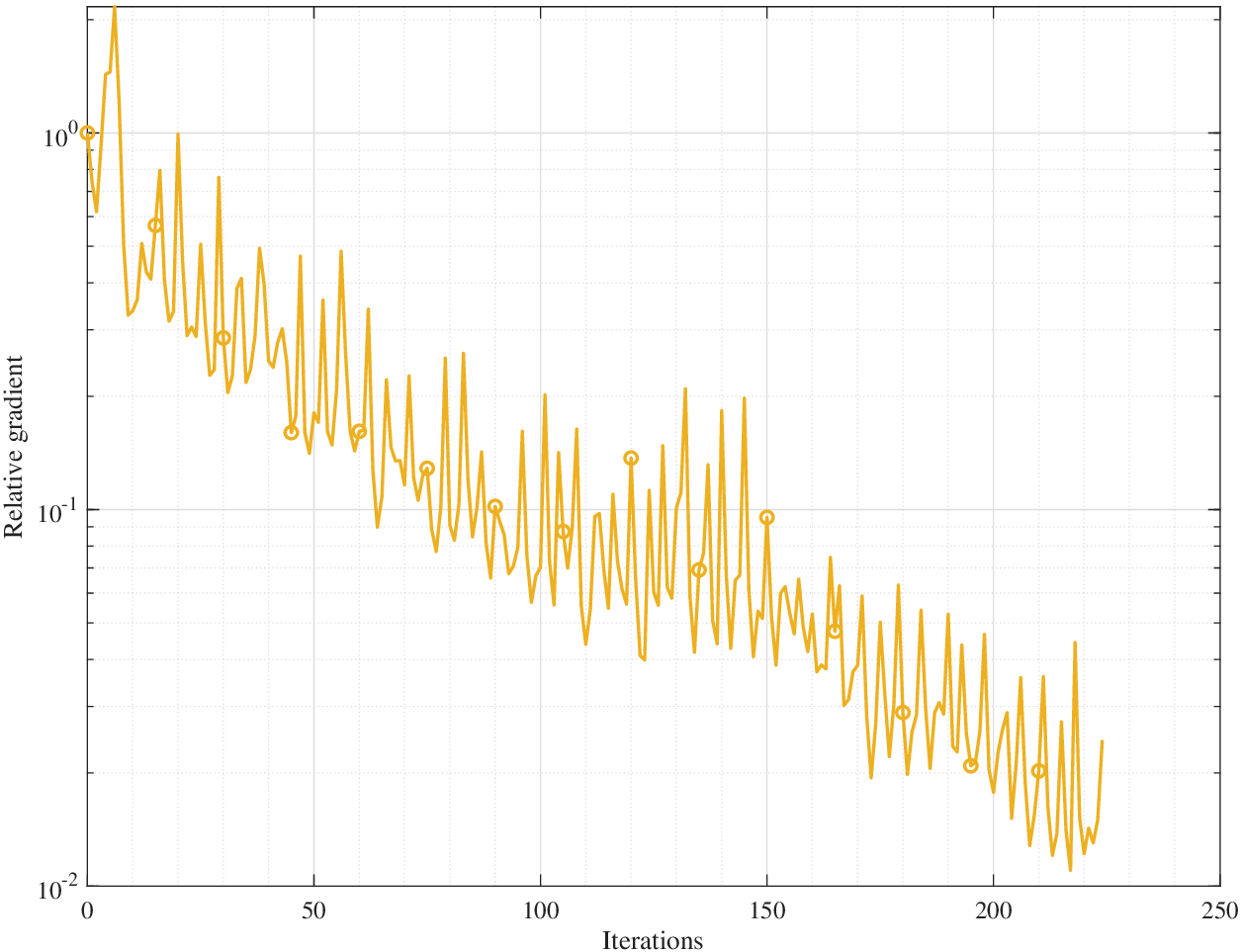}
        \caption*{(b)}
    \end{minipage}
    \caption{Convergence history of the LRBFGS algorithm: (a) objective function value ($\mathcal{H}_2$ error squared); (b) relative Riemannian gradient norm $\|\text{grad} f(x_k)\| / \|\text{grad} f(x_0)\|$.}
    \label{fig1}
\end{figure}

The convergence behavior of Algorithm \ref{alg1} is illustrated in Fig.\ref{fig1}.
The algorithm exhibits a stable and efficient descent trajectory, converging after $225$ iterations.
At termination, the relative Riemannian gradient norm falls below $10^{-2}$.
Notably, the algorithm reduces the $\mathcal{H}_2$-norm from $1.6\times10^{-1}$ provided by BT to $8.5\times10^{-2}$, achieving a significant enhancement.
With the systems driven by a two-dimensional input signal $u(t) = [t^2 e^{-0.2t}, \ 0.5 \cos(\pi t) + 1]$, time-domain comparision of the original and reduced-order systems is provided.
The corresponding output trajectories and the relative errors are depicted in Fig.\ref{fig2}.
Shown in Fig.\ref{fig2}(a), both methods provide reduced-order models that accurately capture the dynamic behavior of the FOM.
However, Fig.\ref{fig2}(b) clearly indicates that the model optimized by our proposed algorithm almost consistently maintains a lower error. 
\par

\begin{figure}[ht]
    \centering
    \begin{minipage}{0.45\textwidth}
        \centering
        \includegraphics[width=\textwidth]{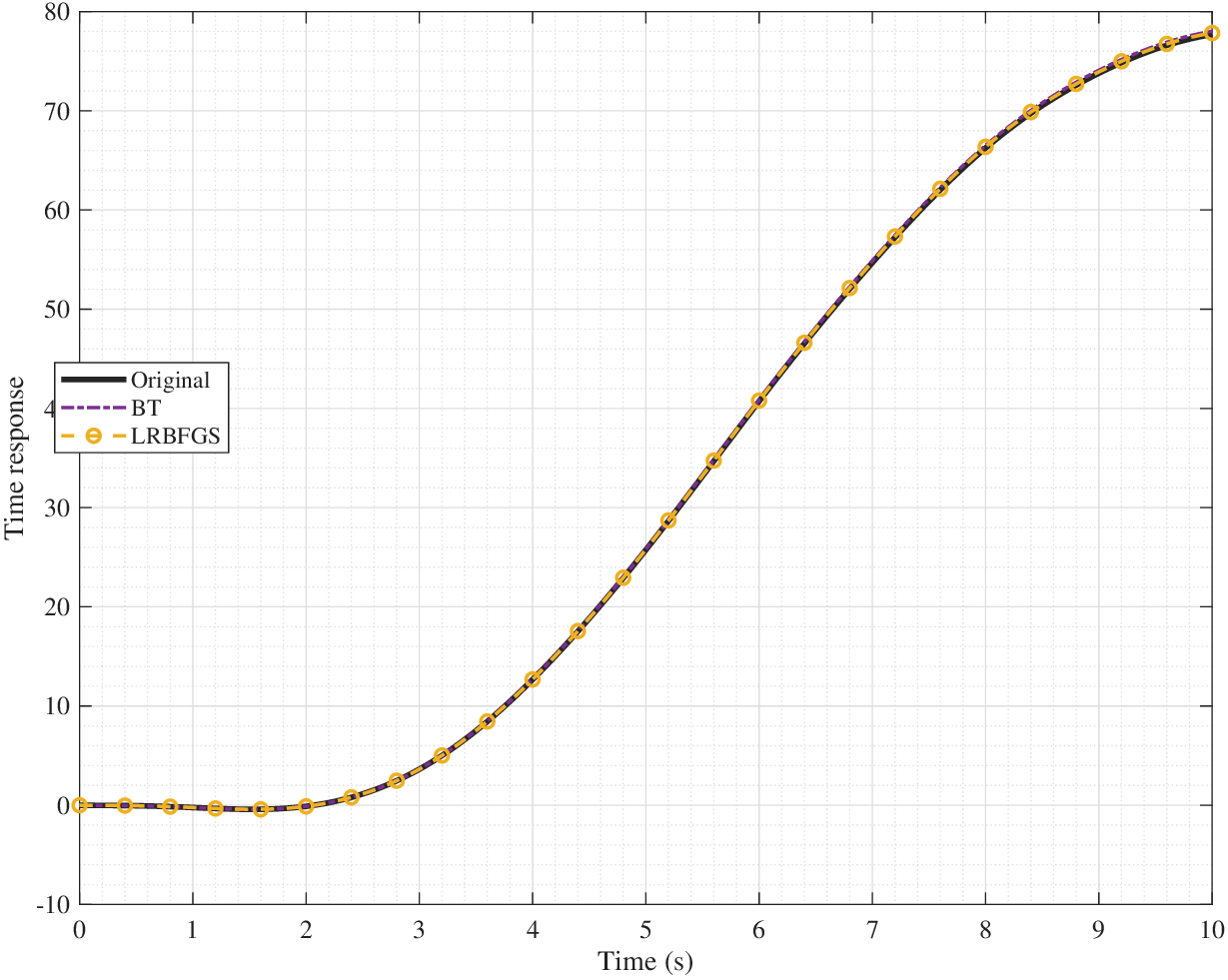}
        \caption*{(a)}
    \end{minipage}
    \hfill
    \begin{minipage}{0.45\textwidth}
        \centering
        \includegraphics[width=\textwidth]{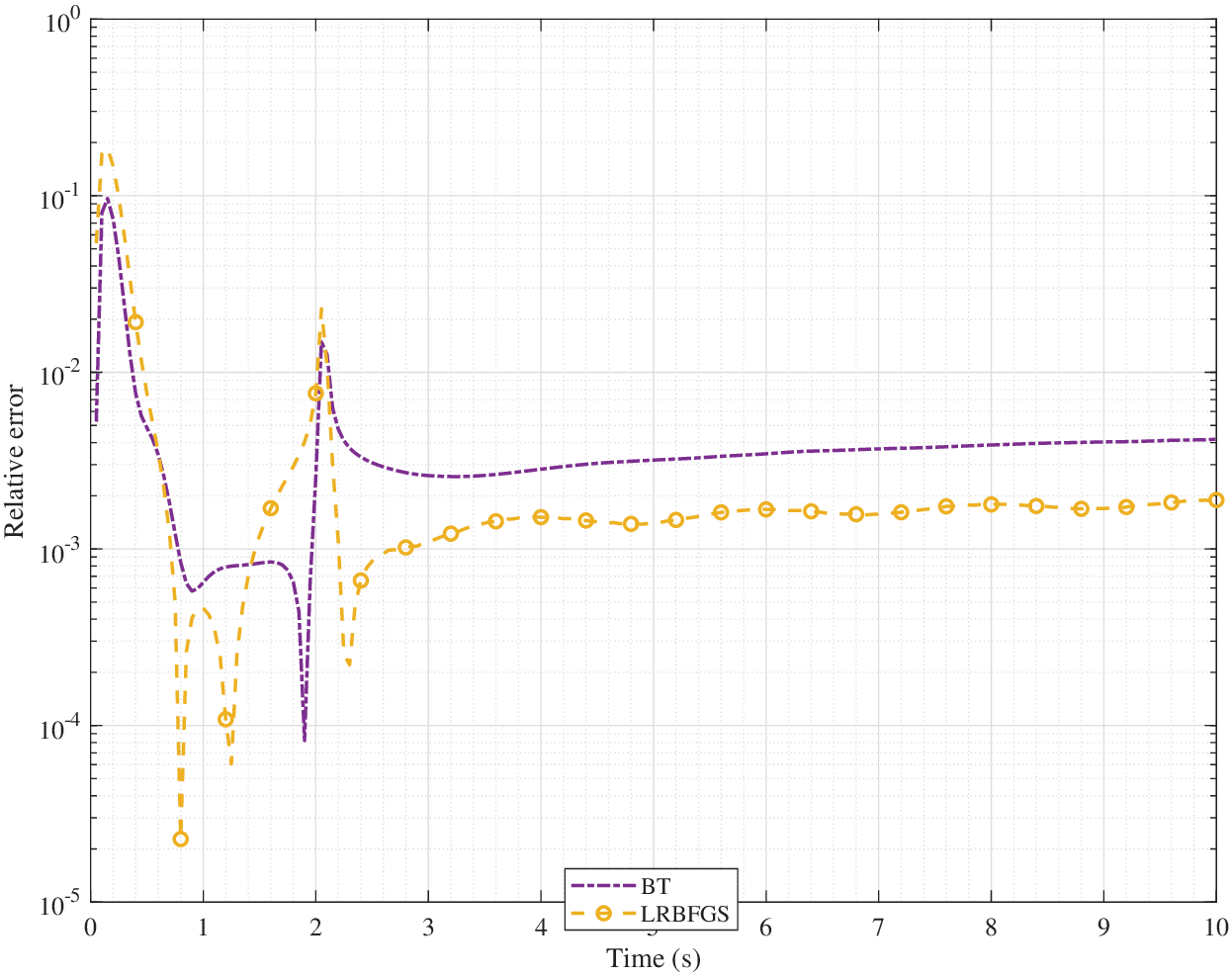}
        \caption*{(b)}
    \end{minipage}
    \caption{Performance comparison of the reduced-order models: (a) transient time responses of the original and reduced systems; (b) instantaneous relative error $|y(t) - \hat{y}(t)| / |y(t)|$.}
    \label{fig2}
\end{figure}

We also consider the cases where the reduced order is set to $r=6$ and $r=14$, with the same parameters and stopping criteria.
The $\mathcal{H}_2$-norms of the corresponding reduced models, alongside the previously discussed case of $r=10$, are summarized in Tabel.\ref{tab1}.
In all evaluated cases, our method is capable of further minimizing the $\mathcal{H}_2$ error, yielding ROMs with superior accuracy.
\par

\begin{table}[ht]
    \centering
    \caption{Comparison of $\mathcal{H}_2$ errors of models obtained via BT and LRBFGS under different reduced orders $r$.}
    \label{tab1}
    \begin{tabular}{c c @{\hspace{0.8cm}} c}
        \toprule
        Reduced order ($r$) & BT & LRBFGS \\
        \midrule
        $6$  & $6.7018 \times 10^{-1}$ & $3.0167 \times 10^{-1}$ \\
        $10$ & $1.6080 \times 10^{-1}$ & $8.5226 \times 10^{-2}$ \\
        $14$ & $3.0300 \times 10^{-2}$ & $1.8339 \times 10^{-2}$ \\
        \bottomrule
    \end{tabular}
\end{table}

\section{Conclusion}

We have studied the $H_2$-optimal reduction of LQO systems based on the production 
manifold. Minimizing the $H_2$ error of MOR is formulated as a Riemannian 
optimization problem, and the Riemannian BFGS method is employed to solve the 
optimization problem iteratively. The resulting reduced models are stable and 
preserve the quadratic structure of the original systems. 
As the coefficient matrices of reduced models are 
selected directly as the optimization variables, the amount of variables is reduced 
dramatically compared to the existing projection methods. The simulation results 
indicate that our approach can provide accurate approximation to high-order systems.

\addcontentsline{toc}{section}{Reference}
\markboth{Reference}{}
\bibliographystyle{elsarticle-num-names}
\bibliography{reference}

\end{document}